\documentclass[reqno,11pt]{amsart}

\usepackage{euscript}
\usepackage{enumerate}
\usepackage{amsmath}
\usepackage{mathrsfs}
\usepackage{mathtools}
\usepackage{verbatim}
\usepackage[initials]{amsrefs}
\usepackage{amssymb}
\usepackage{comment}
\usepackage{color}

\theoremstyle{plain}
\newtheorem{thm}{Theorem} [section]
\newtheorem{prop}[thm]{Proposition}
\newtheorem{cor}[thm]{Corollary}

\newtheorem{lemma}[thm]{Lemma}

\theoremstyle{remark}
\newtheorem{rem}[thm]{Remark}

\theoremstyle{definition}
\newtheorem{defi}[thm]{Definition}
\newtheorem{example}[thm]{Example}

 \newcommand\Cpx{{\mathbb C}}
 \newcommand\Dbb{{\mathbb D}}
 \newcommand\Dc{{\mathcal{D}}}

 \newcommand\Mcal{{\mathcal{M}}}

 \newcommand\Nc{{\mathcal{N}}}

 \newcommand\restrict{{\upharpoonright}}

\newcommand\DT{\operatorname{DT}}
\newcommand\eps{\epsilon}
\newcommand\Nats{{\mathbf N}}
\newcommand\UT{{\mathcal{U}\mathcal{T}}}

\newcount\theTime
\newcount\theHour
\newcount\theMinute
\newcount\theMinuteTens
\newcount\theScratch
\theTime=\number\time
\theHour=\theTime
\divide\theHour by 60
\theScratch=\theHour
\multiply\theScratch by 60
\theMinute=\theTime
\advance\theMinute by -\theScratch
\theMinuteTens=\theMinute
\divide\theMinuteTens by 10
\theScratch=\theMinuteTens
\multiply\theScratch by 10
\advance\theMinute by -\theScratch

\def\today{{\number\day\space
 \ifcase\month\or
  January\or February\or March\or April\or May\or June\or
  July\or August\or September\or October\or November\or December\fi
 \space\number\year}}

\newcommand\HEu{{\EuScript H}}                   

\newcommand\Xt{{\widetilde X}}
\newcommand\Ncal{{\mathcal{N}}}

\newcommand\norm[1]{\ensuremath{\left\vert\left\vert #1 \right\vert\right\vert}}

\begin{document}

\title[]{Some non-spectral  DT-operators in finite von Neumann algebras}

\author[Dykema]{Ken Dykema$^*$}
 \address{Ken Dykema, Department of Mathematics, Texas A\&M University, College Station, TX, USA.}
 \email{ken.dykema@math.tamu.edu}
 \thanks{\footnotesize ${}^{*}$ Research supported in part by NSF grant DMS--1800335.}

\author[Krishnaswamy-Usha]{Amudhan Krishnaswamy-Usha$^{*,\dagger}$}
\address{Amudhan Krishnaswamy-Usha, Delft Institute of Applied Mathematics, Delft University of Technology,
Delft, The Netherlands }
\email{A.KrishnaswamyUsha@tudelft.nl}
\thanks{\footnotesize ${}^{\dagger}$ Portions of this work are included in the thesis of A.Krishnaswamy-Usha for partial fulfillment of the requirements to obtain a Ph.D. degree at Texas A\&M University.}



\begin{abstract} 
Given a DT-operator $Z$ whose Brown measure is radially symmetric and has a certain concentration property, it is shown that $Z$ is not spectral in the sense of Dunford.
This is accomplished by showing that the angles between certain complementary Haagerup-Schultz projections of $Z$ are zero.
New estimates on norms and traces of powers of algebra-valued circular operators over commutative C$^*$-algebras are also proved. 
\end{abstract}

\date{\today}

\maketitle

\section{Introduction}

In~\cite{DKU19}, we investigated, for elements of finite von Neumann algebras, the relationship between
\textit{spectrality} of an operator, in the sense of Dunford~\cite{D54}, and the angles between certain invariant subspaces of the operator.
We thereby showed that the circular free Poisson operators, introduced in~\cite{DH01}, cannot be spectral.
This class includes Voiculescu's circular operator.

The aforementioned invariant subspaces of the operator are the {\em Haagerup--Schultz} subspaces~\cite{HS09}.
The circular free Poisson operators belong to a class of operators called the DT-operators, introduced in~\cite{DH04a}.
In this paper, we will investigate angles between Haagerup--Schultz subspaces of a different set of DT-operators, whose Brown measures satisfy certain properties, and thereby show that also these operators cannot be spectral.
More specifically, those whose Brown measures are radially symmetric but which have ``blow-ups'' of a particular sort (see Theorem~\ref{thm:unza_measures})
are not spectral.
Our proof uses new estimates on norms, conditional expectations and traces of powers of algebra-valued circular operators.

This paper is organized as follows: Section~\ref{sec:prelim} contains a brief review of some elements of spectral operators and Haagerup-Schultz projections.
Section~\ref{sec:angles} is a summary of our previous work on the relationship between spectrality and the angles between Haagerup-Schultz projections.
Section~\ref{sec:DTops} contains an introduction to DT-operators and proof of an upper triangular model for DT-operators.
In Section~\ref{sec:Bcirc} we consider $B$-valued circular operators when $B$ is a commutative C$^*$-algebra and prove some norm estimates involving them.
In Section~\ref{sec:nonspec}, we estimate the angles between certain Haagerup-Schultz subspaces for suitable DT-operators, and use this to show
that they cannot be spectral. 
 
\section{Preliminaries}
\label{sec:prelim}

\subsection{Spectral operators}

Spectral operators were introduced by Dunford \cite{D54} as a generalization of normal operators. Normal operators have a associated projection-valued spectral measure which behaves well with respect to the spectrum. Spectral operators are those which have a well-behaved idempotent-valued spectral measure. To be precise,
 
\begin{defi}
An operator $T$ on a Hilbert space $\HEu$ is said to be {\em spectral} if there exists a map $E$ which sends Borel subsets of
$\Cpx$ to operators in $\HEu$, which satisfies the following criteria:
 
 \begin{enumerate}[(i)]
     \item $ E(\Cpx) = 1 $
     \item $ E(B_1 \cap B_2) = E(B_1) E(B_2)$, for all Borel sets $B_1,B_2$.
     \item $ E( \cup_{i=1}^\infty B_i ) = \sum_{i=1}^\infty E(B_i)$ whenever $B_i$ are pairwise disjoint Borel sets in $\Cpx$, and the sum on the right converges in the strong operator topology.
     \item $\sup\{\norm{E(B)} : B \text{ Borel in } \Cpx\} < \infty$.
     \item $E(B)T=TE(B)$ for all Borel sets $B$.
     \item For all Borel subsets $B$ of $\Cpx$, the spectrum of $T$ restricted to the range of $E(B)$ is contained in the closure of $B$:
        $$ \sigma( T \restrict _{E(B) \HEu}) \subseteq \overline{B}. $$
 \end{enumerate}

\end{defi}

Dunford showed that spectral operators have the following equivalent characterisation. This is a generalization of the well-known result that every matrix is similar to sum of a nilpotent matrix and a diagonal matrix that commute.

\begin{prop}[Dunford]
The following are equivalent:
\begin{enumerate}
    \item $T \in B(\HEu)$ is a spectral operator.
    \item There exist operators $S,N,Q \in B(\HEu)$ such that $S$ is invertible, $N$ is normal, $Q$ is quasinilpotent, $NQ=QN$, and 
    $$ S T S^{-1} = N + Q. $$
\end{enumerate}
\end{prop}

\subsection{Brown measure and Haagerup-Schultz projections}

Throughout the rest of this paper, $(\Mcal,\tau)$ will refer to a von Neumann algebra equipped with a faithful normal tracial state $\tau$.
  
If $N \in \Mcal$ is a normal operator, the projection-valued measure associated with $N$ gives rise to a probability measure on $\Cpx$, given by taking the trace of the corresponding projections.  
L.~ Brown in~\cite{B83} constructed a probability measure $\mu_T$ on $\Cpx$ associated to an arbitrary operator $T$ in $\Mcal$.
The Brown measure $\mu_T$ is just the trace of spectral measure if $T$ is normal.
The Brown measure behaves well with respect to the holomorphic functional calculus of $T$.
For details of the construction, and further properties, see~\cite{B83} and~\cite{HS07}.
  
In~\cite{HS09}, Haagerup and Schultz constructed invariant projections for arbitrary operators in a finite von Neumann algebra.
These subspaces implement a sort of decomposition according to Brown measure.
  
\begin{thm}[\cite{HS09}]\label{thm:hs-proj}
Let $T$ be an operator in a finite von Neumann algebra $(\Mcal,\tau)$ and $B\subseteq \Cpx$ be a Borel set. Then there exists a unique projection $P(T,B) \in \Mcal$ such that the following holds:
\begin{enumerate}[(i)]
\item $\tau(P(T,B))=\mu_T(B)$.
\item $P(T,B)\HEu$ is an invariant subspace for $T$. 
\item The Brown measure of $TP(T,B)$, when computed in the corner \linebreak
$P(T,B)\Mcal P(T,B)$, is concentrated in $B$.
\item The Brown measure of $(1-P(T,B))T$, when computed in the corner $(1-P(T,B))\Mcal (1-P(T,B))$, is concentrated in $B^c$.
\item\label{it:hs-invt} If $Q$ is any $T$-invariant projection with the property that the Brown measure of $TQ$ (computed in $Q\Mcal Q$) is concentrated in $B$, then $Q\leq P(T,B)$.
\end{enumerate}
\end{thm}
  
For discs and complements of discs, the Haagerup-Schultz projections have the following explicit form (see~\cite{HS09}):
  
\begin{prop}\label{prop:hsproj-disc}
Let $T\in (\Mcal,\tau)$, and $r>0$. Let $r \Dbb$ denote the open disc of radius $r$ centered at $0$. Then
  
\begin{gather*}
P(T,\overline{r\Dbb})\HEu = \left \lbrace \xi \in \HEu: \exists  \xi_n \in \HEu, \lim_n \xi_n = \xi, \limsup_n \norm{T^n\xi_n}^{1/n} \leq r \right \rbrace \\
P(T,(r\Dbb)^c)\HEu =\left\lbrace \eta \in \HEu: \exists \, \eta_n \in \HEu, \lim_{n \to \infty} T^n \eta_n = \eta, \limsup_n \norm{\eta_n}^{1/n} \leq \frac{1}{r}  \right \rbrace_.
\end{gather*}
\end{prop}

 The Haagerup-Schultz projections also respect lattice operations (see~\cite{Sch06} or Section~3 of~\cite{CDSZ17}).
 
\begin{prop}\label{prop:hsproj-lattice}
Let $T \in (\Mcal,\tau)$, and $B_1,B_2,\ldots \subseteq \Cpx$ be Borel sets. Then
\begin{align*}
P(T, \bigcup_{i\ge1}B_i ) &= \bigvee_{i\ge1}P\left(T,B_i\right) \\
P(T, \bigcap_{i\ge1} B_i ) &= \bigwedge_{i\ge1} P\left(T,B_i\right).
\end{align*}
\end{prop}
  
Note that the above relations imply that $P(T,B)\wedge P(T,B^c)=0$ and $P(T,B)\vee P(T,B^c) = 1$. 
In general, $P(T,B)$ and $P(T,C)$ do not commute, unless one of $B$ and $C$ is a subset of the other.

The projections also behave well with respect to compressions and similarities (see Theorems 2.4.4 and~12.3 in~\cite{CDSZ17}).

\begin{thm}\label{prop:hsproj-sim}
Let $Q \in \Mcal$ be a non-zero, $T$-invariant projection and let $A\in \Mcal$ be invertible.
Then, for all Borel sets $B \subseteq \Cpx$, we have 
\begin{enumerate}[(i)]
\item $P(T,B) \wedge Q = P^{(Q)}(TQ,B)$,
\item  $\mu_{ATA^{-1}} = \mu_T$,
\item $P(ATA^{-1},B)\HEu = \overline{AP(T,B)\HEu}$,
\end{enumerate}
where $P^{(Q)}$ denotes the Haagerup-Schultz projection computed in the compression $Q\Mcal Q$. 
\end{thm}

Haagerup and Schultz also showed the following equivalence: 
\begin{prop}[\cite{HS09}]\label{prop:sotqn}
 Let $T\in (\Mcal,\tau)$.
 Then the Brown measure $\mu_T$ has support equal to $\{0\}$ if and only if 
$$ \text{s.o.t.-}\lim_{n\to \infty} (T^{*n} T^n)^{1/2n} = 0, $$
where the limit is in the in the strong operator topology.
\end{prop}

By analogy with the quasinilpotent case, we call an operator with Brown measure concentrated at $\{0\}$ an s.o.t-quasinilpotent operator.

\section{Angles between projections and spectrality}
\label{sec:angles}
 
In previous work \cite{DKU19}, we showed a relationship between the angles between Haagerup-Schultz projections and the spectrality of the operator.
 
\begin{defi}
For closed, nonzero subspaces $V$ and $W$ of a Hilbert space $\HEu$, we define the {\em angle} between them to be
$$ \alpha(V,W) = \inf \{ \text{angle}(v,w): v \in V,\, w \in W,\, v \neq 0,\, w \neq 0\}. $$
For nonzero projections $P,Q \in B(\HEu)$, $\alpha(P,Q)$ will denote the angle between the spaces $P\HEu$ and $Q\HEu$.
\end{defi}
From consideration of the unital C$^*$-agebra generated by two projections (see~\cite{RS89} or~\cite{BS10}) it is known that the cosine of the angle between $P\HEu$ and $Q\HEu$ is equal to the maximum element of the spectrum of $PQP$.
Thus, $\alpha(P,Q)$ is well defined for projections $P$ and $Q$ in a unital C$^*$-algebra, independently of the way the C$^*$-algebra is unitally represented on a Hilbert space.
\begin{defi}
For an operator $T \in \Mcal$, let $\alpha(T)$ denote the infimum of the angles between complementary Haagerup-Schultz projections of $T$.
That is,
$$ \alpha(T)\coloneqq\inf_{\substack{B \subseteq \Cpx,\,B \text{ Borel},\\  0<\mu_T(B)<1}} \alpha(P(T,B),P(T,B^c)). $$ 
We may also write $\alpha_\Mcal(T)$ for $\alpha(T)$, in order to emphasize the von Neumann algebra we consider.
We say $T$ has the {\em uniformly non-zero angle property} (or UNZA property) if $\alpha(T)>0$.
\end{defi}
In \cite{DKU19}, we showed that operators having the UNZA property are those that are similar to the sum a of a normal operator and a commuting s.o.t.-quasinilpotent operator.
\begin{thm}[\cite{DKU19}]\label{thm:angle-sotqn}
Let $T\in \Mcal$. Then the following are equivalent. 
\begin{enumerate}
\item $T$ has the UNZA property.
\item There exist $S, Q, N \in \Mcal$, with $S$ invertible, $NQ=QN$, $N$
 normal, $Q$ s.o.t.quasinilpotent, and 
$$ S T S^{-1} = N + Q. $$
\end{enumerate}  
\end{thm}

\begin{cor}[\cite{DKU19}]
Let $T\in\Mcal$.
Then $T$ is spectral if and only if $T$ is decomposable and has the UNZA property.
\end{cor}

\section{DT-operators}
\label{sec:DTops}
 
We now turn to the study of DT-operators, which were
introduced in~\cite{DH04a}.
These are elements of a von Neumann algebra $\Mcal$ equipped with a normal faithful tracial state $\tau$,
of the form $Z=D+cT\in\Mcal$, where $D$ belongs to a copy $\Dc$ of $L^\infty[0,1]$ in $\Mcal$
(whereon $\tau$ is given by integration with respect to Lebesgue measure), $c>0$  and $T$ is
the quasinilpotent DT-operator, that is constructed in a particular way from $\Dc$ and a semicircular element $X$
that is free from $\Dc$, by ``cutting out the upper triangle'' of $X$.
We call $Z$ a $\DT(\mu,c)$-operator, where $\mu$ is the distribution of $D$ with respect to $\tau$.
Moreover, any element of a tracial von Neumann algebra having the same $*$-moments as $Z$ is called a $\DT(\mu,c)$-operator.
The DT-operators also arise as limiting $*$-distributions of certain upper triangular random matrices.

The following is is a synthesis of results taken from~\cite{DH04a}.
\begin{thm}\label{thm:d+t}
Let $X \in (\Mcal,\tau)$ be a semicircular operator with $\tau(X^2)=1$, $\tau(X)=0$, and let $$\lambda:L^\infty([0,1]) \to \Mcal$$ be a normal, unital, injective $*$-homomorphism such that $X$ and the image of $\lambda$ are free with respect to $\tau$. 
Then, there exists an operator $T \in \Mcal$, constructed from $X$ and $\lambda$ is a prescribed way and written $T=\UT(X,\lambda)$, such that the following holds:
\begin{enumerate}[(i)]
\item For all $0<t<1$, 
$$ T \lambda(1_{[0,t]}) = \lambda(1_{[0,t]}) T \lambda(1_{[0,t]}). $$
\item $X = T + T^*$.
\item For all $0<t<1$, $\lambda(1_{[0,t]})T\lambda(1_{[t,1]}) = \lambda(1_{[0,t]})X\lambda(1_{[t,1]})$.
\item If $f \in L^\infty([0,1])$, $c>0$, and $D=\lambda(f)$, then $D+cT$ is a $\DT(\mu,c)$ operator, where $\mu$ is the push-forward of the Lebesgue measure by $f$.
\item $T$ is quasinilpotent.
\end{enumerate}
\end{thm}

We call $T$ the {\em quasinilpotent $\DT$-operator}.

The following is a generalization of part of Theorem~4.12 of~\cite{DH04a}.
 
\begin{thm}\label{thm:dt-ut}
 Let $c>0$, and $\mu,\mu_1,..\mu_n$ be compactly supported Borel probability measures on $\Cpx$ such that $\mu$ is a convex combination of $\mu_i$:
$$ \mu =\sum_{i=1}^n t_i \mu_i $$
for some $0<t_i<1$ such that $\sum_it_i=1$.
Then there is an example of a $\DT(\mu,c)$-operator $Z$ in a von Neumann algebra $\Mcal$
with the following properties:
\begin{enumerate}[(i)]
\item $\Mcal$ has a $*$-subalgebra $\Nc$ and a semicircular element $X$ normalized so that $\tau(X^2)=1$ and such that $\Nc$ and $\{X\}$ are free.
\item There are projections $p_1,\ldots,p_n$ in $\Nc$ so that $\tau(p_j)=t_j$ and $p_1+\cdots+p_n=1$.
\item For each $j$, there is a $\DT(\mu_j,c\sqrt{t_j})$-operator $a_j\in p_j\Nc p_j$, with respect to the renormalized trace $t_i^{-1}\tau$
\item Letting $b_{ij}=cp_iXp_j$ for each  $i<j$,
\[
Z=\sum_{j=1}^n a_j+\sum_{1\le i<j\le n}b_{ij}
=\begin{pmatrix}
a_1 & b_{12} & \cdots & \cdots & b_{1n} \\
0 & a_2 & b_{23} & & \vdots \\
\vdots & \ddots & \ddots & \ddots & \vdots\\
\vdots && 0 & a_{n-1} & b_{n-1,n} \\[1.5ex]
0& \cdots & \cdots & 0 & a_n
\end{pmatrix}_,
\]
where  we have used the natural matrix notation with respect to the projections $p_1,\ldots,p_n$.
\end{enumerate}

Furthermore, if there is a Borel set $B\subseteq\Cpx$ so that $\mu_1$ is  concentrated in $B$
and, for each $2\le j\le n$, $\mu_j$ is concentrated in $B^c$,
then $p_1=P(Z,B)$.

\end{thm}

The proof uses the following result, which follows immediately from Voicu\-lescu's random matrix results~\cite{V91}
by an argument, pioneered by Voiculescu, that is well known and which we omit.

\begin{lemma}\label{lem:semic}

Suppose $\Dc$ is a copy of $L^\infty([0,1])$ in $\Mcal$
and
$X,\tilde{X}$ are semicircular elements in $\Mcal$, with $\tau(X^2)=\tau(\tilde{X}^2)=1$, such that $(\{X\},\{\tilde{X}\},\Dc)$ is a free family. 
Let $p_1,\ldots,p_N\in\Dc$ be projections with $\sum p_i = 1$.
Then
\[ 
Y=\sum_{i=1}^N p_i X p_i + \sum_{i<j} \left( p_i \tilde{X} p_j + p_j \tilde{X} p_i \right)
\] 
is a semicircular element in $\Mcal$, with $\tau(Y^2)=1$.
Moreover, $\{Y\}$ and $\Dc$ are free.
\end{lemma}
 
\begin{proof}[Proof of Theorem~\ref{thm:dt-ut}]
This imitates some of the proof of Theorem 4.12 in \cite{DH04a}, using Lemma \ref{lem:semic}.
Take a von Neumann algebra $\Mcal$ with a trace-preserving, normal, injective *-homomorphism, $\lambda:L^\infty([0,1])\to \Mcal$
and having
$X,\Xt \in \Mcal$, centered semi-circular elements such that $\tau(X^2)=\tau(\Xt^2)=1$ and
\[
\{ X\},\,\{\Xt\} ,\lambda(L^\infty([0,1]))
\]
is a free family.
Choose $0=s_0<s_1<\cdots s_{n}=1$ such that $s_{i}-s_{i-1}=t_i$, for $1\leq i \leq n$, and let $p_i = \lambda(1_{[s_{i-1},s_i]})$.
Then the $p_i$ are projections and $\tau(p_i)=t_i$.
From Lemma \ref{lem:semic}, 
\[ Y = \sum_{i=1}^n p_i \Xt p_i + \sum_{i \neq j} p_i X p_j \]
 is a semicircular element with $\tau(Y^2)=1$, and  $Y$ is free from $\lambda(L^\infty([0,1])$.
For $1\le i\le n$, let $\lambda_i:L^\infty([0,1])\to p_i\Mcal p_i$ be the unital $*$-homomorphism given by $\lambda_i(f)=\lambda(g)$ where
\[
g(x)=\begin{cases}
f((x-s_{i-1})/t_i),& s_{i-1}\le x\le s_i \\
0,&\text{otherwise.}
\end{cases}
\]
Let $f_i\in L^\infty([0,1])$ be such that the push-forward under $f_i$ of Lebesgue measure is $\mu_i$ and let $f\in L^\infty([0,1])$ be defined by
\[
\lambda(f)=\lambda_1(f_1)+\cdots+\lambda_n(f_n).
\]
Then the push-forward of Lebesgue measure under $f$ is $\mu$.
Let $T=\UT(Y,\lambda)$ be the operator from Theorem \ref{thm:d+t} corresponding to $Y$ and $\lambda$.
Let $Z=\lambda(f)+T$.
Then $Z$ is a $\DT(\mu,1)$-operator.
Let $\Nc$ be the von Neumann algebra generated by $\lambda(L^\infty([0,1]))$ and $\Xt$, so that $\Nc$ and $\{X\}$ are free.
Let  $a_i=p_iZp_i=\lambda_i(f_i)+p_iTp_i$ and, for $i<j$, let $b_{ij}=p_iTp_j$.
Note that $p_iTp_i=p_i\UT(\Xt,\lambda)p_i\in\Nc$ and $b_{ij}=p_iXp_j$.
By Lemma~4.10 of~\cite{DH04a},  $p_iTp_i=\UT(p_i\Xt p_i,\lambda_i)$,.
We have that $\lambda_i(L^\infty([0,1]))$ and $p_i\Xt p_i$ are free.
Moreover, $p_i\Xt p_i$ is with respect to the trace $\tau(p_i)^{-1}\tau\restrict_{pi\Ncal p_i}$ a
semicircular element with second moment $\tau(p_i)=t_i$.
It follows that $a_i$ is a $\DT(\mu_i,\sqrt{t_i})$-element in $p_i\Ncal p_i$ with respect to this trace.
Thus, we have
\[ D+c T = \sum_{i=1}^n a_i + \sum_{1<i<j<n} b_{ij} \]
as required.

For the last statement, since the Brown measure of $a_1$ is concentrated in $B$, it follows from Theorem~\ref{thm:hs-proj}(\ref{it:hs-invt})
that $p\le P(Z,B)$.
However, $\tau(p)=t_1=\mu(B)=\tau(P(Z,B))$, so we have equality.

\end{proof}
 
The following result is a consequence of the upper-triangular decomposition theorem and the properties of Haagerup--Schultz projections that are described in Theorem~\ref{thm:hs-proj}.
 
\begin{thm}\label{thm:dt-properties}
Let $T$ be a $\DT(\mu,c)$ operator. Then,
\begin{enumerate}[(i)]
\item The Brown measure of $T$ is $\mu$.
\item $T$ is a decomposable operator, and $\sigma(T)=\text{supp}(\mu)$.
\item If $B\subseteq \Cpx$ is a Borel set with $\mu(B)\neq 0$, and $p=P(T,B)$ is the corresponding Haagerup-Schultz projection, then $Tp$ is also a DT-operator. To be precise,
$Tp$ is a $\text{DT}(\mu(B)^{-1}\mu\restrict_B, c\sqrt{\mu(B)})$ operator in $p\Mcal p$.
\end{enumerate}
\end{thm}

\section{$B$-valued circular operators where $B$ is commutative}
\label{sec:Bcirc}

Let $B$ be a unital C$^*$-algebra.
A $B$-valued C$^*$-noncommutative probability space is a pair $(A,E)$, where $A$ is a unital C$^*$-algebra containing a unital copy of $B$ and where $E:A\to B$
is a conditional expectation.
Given $a\in A$, the {\em $B$-valued $*$-moments} of $a$ are the maps that catalog all of the information about the values of expectations
of the form $E(a^{\eps(1)}b_1a^{\eps(2)}\cdots b_{n-1}a^{\eps(n)})$,
for $n\in\Nats$, $\eps(1),\ldots,\eps(n)\in\{1,*\}$ and $b_1,\ldots,b_{n-1}\in B$.
Speicher~\cite{Sp98} developed the theory of {\em $B$-valued cumulants} in terms of which the $B$-valued $*$-moments of elements of $A$ can be written.
The notion of a {\em $B$-valued circular operator} was introduced in~\cite{SS01}.
It is an element $a\in A$ such that the $B$-valued cumulants of $(a,a^*)$ vanish except for the balanced ones of order $2$.
(See~\cite{SS01} or~\cite{BD18} for details).

This section is directed toward making some norm estimates involving $B$-valued circular operators where $B$ is commutative.
It ends with an application to $\DT$-operators.

It is important to note that the quasinilpotent DT-operator $T$ is $B$-valued circular for $B=C([0,1])$.
In~\cite{S03}, Piotr {\'S}niady proved some important results about $T$,
or, more correctly, about a certain $B$-valued circular operator for $B=L^\infty([0,1])$ (or, equally well, for $B=C([0,1])$).
(A full proof that $T$ is $B$-valued circular is found in the appendix of~\cite{DH04b}.)
Here, $B$ is identified with $\lambda(C([0,1]))$ via $\lambda$ and we
consider the $B$-valued C$^*$-noncommutative probability space with $\tau$-preserving conditional expectation from the C$^*$-algebra generated by $X$ and $B$ onto $B$.

Throughout this section $B$ will be a commutative, unital C$^*$-algebra and $(A,E)$ a $B$-valued C$^*$-noncommutative probability space.
A $B$-valued circular element is $T\in A$ with $E(T)=0$
and such that the only nonvanishing noncrossing cumulants of the pair $(T_1,T_2)=(T,T^*)$ are the completely positive maps
from $B$ to $B$, given by
$\alpha_{1,2}(b)=E(TbT^*)$ and $\alpha_{2,1}(b)=E(T^*bT)$.
These are positive maps from $B$ to $B$ and must satisfy
\begin{equation}\label{eq:|b|}
|\alpha_{1,2}(b)|\le\alpha_{1,2}(|b|),\qquad|\alpha_{2,1}(b)|\le\alpha_{2,1}(|b|).
\end{equation}
Indeed, if $\phi$ is a state of $B$, then $\phi\circ\alpha_{1,2}$ is a positive linear functional on $B$ and is given by integration with respect to a measure on the spectrum of $B$,
so satisfies
\[
|\phi\circ\alpha_{1,2}(b)|\le\phi\circ\alpha_{1,2}(|b|).
\]

By way of notation, given a finite sequence $\eps(1),\eps(2),\ldots,\eps(n)\in\{1,*\}$, we will say the sequence is {\em balanced} if 
there are as many $*$'s as $1$'s, namely, if
\[
\#\{j: \eps(j)=1\}=\frac n2.
\]
\begin{lemma}\label{lem:EXeps-pos}
Let $B$ be a commutative, unital C$^*$-algebra
and let $T$ be a $B$-valued circular element in some $B$-valued C$^*$-noncommutative probability space $(A,E)$.
Then for all $\eps(1),\ldots,\eps(n)\in\{1,*\}$, we have
\begin{equation}\label{eq:EXeps-pos}
E(T^{\eps(1)}T^{\eps(2)}\cdots T^{\eps(n)})\ge0.
\end{equation}
\end{lemma}
\begin{proof}
If the sequence $\eps(1),\ldots,\eps(n)$ is not balanced, then from the moment-cumulant formula, the expectation on the left hand side of~\eqref{eq:EXeps-pos}
is zero.
So we may assume the sequence is balanced.
Let $\alpha_{1,2}$ and $\alpha_{2,1}$ be the cumulant maps for $T,T^*$.
We proceed by induction on $n$.
The case $n=2$ follows by positivity of the maps $\alpha_{1,2}$ and $\alpha_{2,1}$.

For the induction step, suppose $n\ge4$.
Let $J$ be the set of all $j\in\{2,3,\ldots,n\}$ such that $\eps(j)\ne\eps(1)$ and the sequence $\eps(1),\ldots,\eps(j)$ is balanced.
Suppose for the moment $\eps(1)=1$.
Then by the moment-cumulant formula, we have
\begin{equation}\label{eq:EXeps-sum}
E(T^{\eps(1)}\cdots T^{\eps(n)})=\sum_{j\in J}\alpha_{1,2}\big(E(T^{\eps(2)}\cdots T^{\eps(j-1)})\big)E(T^{\eps(j+1)}\cdots T^{\eps(n)}).
\end{equation}
Using the induction hypothesis, the positivity
of the map $\alpha_{1,2}$ and the commutativity of $B$, we see that for each $j\in J$ the corresponding term in the sum~\eqref{eq:EXeps-sum}
is positive and, hence, we get $E(T^{\eps(1)}T^{\eps(2)}\cdots T^{\eps(n)})\ge0$.

The situation when $\eps(1)=*$ is the same but with $\alpha_{2,1}$ replacing $\alpha_{1,2}$ in~\eqref{eq:EXeps-sum}.
\end{proof}

\begin{lemma}\label{lem:|b|}
Let $B$ be a commutative, unital C$^*$-algebra and suppose $T$ is a $B$-valued circular element in a C$^*$-noncommutative probability space $(A,E)$.
Then for every $n\in\Nats$, $\eps(1),\ldots,\eps(n)\in\{1,*\}$ and $b_1,\ldots,b_n\in B$, we have
\begin{equation}\label{eq:|E|bdd}
\big|E(T^{\eps(1)}b_1T^{\eps(2)}b_2\cdots T^{\eps(n)}b_n)\big|\le\left(\prod_{j=1}^n\|b_j\|\right)E(T^{\eps(1)}T^{\eps(2)}\cdots T^{\eps(n)}).
\end{equation}
\end{lemma}
\begin{proof}
Let $\alpha_{1,2}$ and $\alpha_{2,1}$ be the cumulant maps for the pair $(T_1,T_2)=(T,T^*)$, given by
\[
\alpha_{1,2}(b)=E(TbT^*),
\qquad
\alpha_{2,1}(b)=E(T^*bT).
\]
Note that, if the sequence $\eps(1),\ldots,\eps(n)$ is not balanced, then both sides of~\eqref{eq:|E|bdd} are zero, so we assume the sequence is balanced
and we proceed to prove~\eqref{eq:|E|bdd} by induction on $n$.
For the case $n=2$,
using~\eqref{eq:|b|} we get 
\begin{multline*}
|E(T^{\eps(1)}b_1T^{\eps(2)}b_2)|=|E(T^{\eps(1)}b_1T^{\eps(2)})|\,|b_2| \\
\le E(T^{\eps(1)}|b_1|T^{\eps(2)})\,|b_2|\le\|b_1\|\,\|b_2\|\,E(T^{\eps(1)}T^{\eps(2)}).
\end{multline*}
For the induction step, let $n\ge4$ and suppose $\eps(1)=1$.
Letting $J$ be as in the proof of Lemma~\ref{lem:EXeps-pos}, using the moment-cumulant formula for $T$,
equation~\eqref{eq:|b|} and the induction hypothesis,
we have
\begin{align*}
&\big|E(T^{\eps(1)}b_1T^{\eps(2)}b_2\cdots T^{\eps(n)}b_n)\big| \\
&=\left|\sum_{j\in J}\alpha_{1,2}\big(b_1E(T^{\eps(2)}b_2\cdots T^{\eps(j-1)}b_{j-1})\big)b_jE(T^{\eps(j+1)}b_{j+1}\cdots T^{\eps(n)}b_n)\right| \displaybreak[1] \\[1ex]
&\le\sum_{j\in J}\big|\alpha_{1,2}\big(b_1E(T^{\eps(2)}b_2\cdots T^{\eps(j-1)}b_{j-1})\big)\big|\,\|b_j\|\,\big|E(T^{\eps(j+1)}b_{j+1}\cdots T^{\eps(n)}b_n)\big|
 \displaybreak[1] \\
&\le\sum_{j\in J}\alpha_{1,2}\big(\big|b_1E(T^{\eps(2)}b_2\cdots T^{\eps(j-1)}b_{j-1})\big|\big)\,\left(\prod_{i=j}^n\|b_i\|\right)E(T^{\eps(j+1)}\cdots T^{\eps(n)})
 \displaybreak[2] \\
&\le\sum_{j\in J}\alpha_{1,2}\big(\|b_1\|\big|E(T^{\eps(2)}b_2\cdots T^{\eps(j-1)}b_{j-1})\big|\big)\,\left(\prod_{i=j}^n\|b_i\|\right)E(T^{\eps(j+1)}\cdots T^{\eps(n)}) 
 \displaybreak[1] \\
&\le\sum_{j\in J}\alpha_{1,2}\big(\left(\prod_{i=1}^{j-1}\|b_i\|\right)E(T^{\eps(2)}\cdots T^{\eps(j-1)})\big)\,\left(\prod_{i=j}^n\|b_i\|\right)E(T^{\eps(j+1)}\cdots T^{\eps(n)})
 \displaybreak[1] \\
&=\left(\prod_{i=1}^{n}\|b_i\|\right)\sum_{j\in J}\alpha_{1,2}\big(E(T^{\eps(2)}\cdots T^{\eps(j-1)})\big)\,E(T^{\eps(j+1)}\cdots T^{\eps(n)}) \\
&=\left(\prod_{i=1}^{n}\|b_i\|\right)E(T^{\eps(1)}\cdots T^{\eps(n)}).
\end{align*}
When $\eps(1)=*$, the proof is the same but with $\alpha_{2,1}$ replacing $\alpha_{1,2}$.
\end{proof}

\begin{lemma}\label{lem:norms}
Let $B$ be a commutative, unital C$^*$-algebra and suppose $T$ is a $B$-valued circular element in a C$^*$-noncommutative probability space $(A,E)$
with $E$ faithful.
Then for every $n\in\Nats$and $\eps(1),\ldots,\eps(n)\in\{1,*\}$ and all $b_1,\ldots,b_n\in B$, we have
\[
\|T^{\eps(1)}b_1T^{\eps(2)}b_2\cdots T^{\eps(n)}b_n\|\le\left(\prod_{j=1}^n\|b_j\|\right)\|T^{\eps(1)}T^{\eps(2)}\cdots T^{\eps(n)}\|.
\]
In particular, if $T$ is quasinilpotent and $b_1,b_2\in B$, then $b_1Tb_2$ is quasinilpotent.
\end{lemma}
\begin{proof}
We may without loss of generality assume that $A$ and $B$ are countably generated, and, thus, separable.
Let 
\[
W=T^{\eps(1)}b_1T^{\eps(2)}b_2\cdots T^{\eps(n)}b_n,\qquad V=T^{\eps(1)}T^{\eps(2)}\cdots T^{\eps(n)}
\]
and let $m=\left(\prod_{j=1}^n\|b_j\|\right)$.
For any element $x\in A$, we have
\[
\|x\|=\limsup_{n\to\infty}\|E((x^*x)^n)\|^{1/2n}.
\]
Indeed, if $\phi$ is a faithful state on $B$, then $\phi\circ E$ is a faithful state on $A$, and we have
\[
\|x\|=\limsup_{n\to\infty}\big(\phi\circ E((x^*x)^n)\big)^{1/2n}.
\]
But
\[
\phi\circ E((x^*x)^n)\le \|E((x^*x)^n)\| \le \|x\|^{2n}.
\]

Applying Lemma~\ref{lem:|b|}, for every $n\in\Nats$ we get
\[
E((W^*W)^n)\le m^{2n}E((V^*V)^n)
\]
and we conclude $\|W\|\le m\|V\|$.
\end{proof}

\begin{lemma}\label{lem:Z2norm}
Let $B$ be a commutative, unital C$^*$-algebra and suppose $T$ is a quasinilpotent $B$-valued circular element.

Take $Z=b+T$, where $b\in B$ is invertible.
Then $Z$ is invertible and for all $n\in\Nats$, we have
\[
 E( (Z^{n})^* Z^n ) \ge (b^{n})^* b^n 
\] 
\[
E((Z^{-n})^*Z^{-n})\ge(b^{-n})^*b^{-n}.
\]
\end{lemma}
\begin{proof}

To prove the first inequality, we first expand $Z^n$ as follows: 

\[ 
Z^n = b^n + \sum_{p=1}^n \sum_{\substack{q_0,\ldots,q_p\ge0\\q_0+\cdots+q_p=n-p}}b^{q_0} T b^{q_1}\cdots T b^{q_p}.
\]
Since $T$ is $B$-circular,
\[ 
E(b_0 T b_1 T b_2 \cdots T b_p) = 0
\]
for all $p \in \Nats$ and $b_0,...b_p \in B$. 
If we let $Y_n = Z^n - b^n $, we have
$E(b_0 Y_n) = 0$, for every $b_0 \in B$. 
We then have  
\begin{align*} 
E( (Z^n)^* Z^n )&= E( (b^n)^*b^n) + E( (b^n)^* Y_n) + E(Y_n^*b^n) + E(Y_n^*Y_n) \\
&= (b^n)^* b^n + E(Y_n^*Y_n) \geq (b^n)^* b^n.
\end{align*}

The proof of the second inequality is largely along the same lines, using the power series expansion for $Z^{-1}$. 
We have
$Z=b(1+b^{-1}T)$.
By Lemma~\ref{lem:norms}, $b^{-1}T$ is quasinilpotent.
Therefore, $Z$ is invertible and its inverse has the power series expansion
\[
Z^{-1}=(1+b^{-1}T)^{-1}b^{-1}=b^{-1}+\sum_{k=1}^\infty(-1)^k(b^{-1}T)^kb^{-1},
\]
which converges in norm.
Let $Y_1=Z^{-1}-b^{-1}$ be the summation found above.
By using the power series expansion for $Y_1$ and the fact that $T$ is $B$-circular, we see
\[
E(b_0Y_1b_1Y_1b_2\cdots Y_1b_p)=0
\]
for all $p\in\Nats$ and $b_0,b_1,\ldots,b_p\in B$.
We have
\[
Z^{-n}=b^{-n}+\sum_{p=1}^n\sum_{\substack{q_0,\ldots,q_p\ge0\\q_0+\cdots+q_p=n-p}}b^{-q_0}Y_1b^{-q_1}\cdots Y_1b^{-q_p}.
\]
Let $Y_n=Z^{-n}-b^{-n}$ be the summation found above.
Then we have $E(bY_n)=0$ for every $b\in B$.
Thus, we have
\begin{multline*}
E((Z^{-n})^*Z^{-n})=(b^{-n})^*b^{-n}+E((b^{-n})^*Y_n)+E(Y_n^*b^{-n})+E(Y_n^*Y_n) \\
=(b^{-n})^*b^{-n}+E(Y_n^*Y_n)\ge(b^{-n})^*b^{-n}.
\end{multline*}
\end{proof}

Now we get the following 2-norm estimates for powers of a DT-operator $Z$.
For $0< r<s$, we let $A(r,s)$ be the closed annulus with radii $r$ and $s$.

\begin{cor}\label{cor:Znorms}
Let $Z$ be a DT-operator whose spectrum is contained in the closed annulus $A(r,s)$ and let $k\ge1$ be a natural number.
 Then
\[
\tau((Z^k)^*Z^k) \geq r^{2k},\qquad \tau((Z^{-k})^*Z^{-k})\geq s^{-2k}.
\]
\end{cor}
\begin{proof}
We have $Z=b+cT$ where $T$ is the quasinilpotent DT-operator and $b$ is a normal operator whose spectrum lies in the closed annulus $A(r,s)$.
Since the conditional expectation $E$ is $\tau$--preserving, from Lemma~\ref{lem:Z2norm} we have
\begin{align*}
\tau((Z^{k})^*Z^k)&=\tau\circ E(((Z^{k})^*Z^k) \geq \tau((b^k)^*b^k)\ge r^{2k},\\
\tau((Z^{-k})^*Z^{-k})&=\tau\circ E((Z^{-k})^*Z^{-k})\geq\tau((b^{-k})^*b^{-k})\geq s^{-2k}.
\end{align*}
\end{proof}

\section{Some non-spectral DT-operators} 
\label{sec:nonspec}

In \cite{DKU19}, we showed that all operators in a class which includes Voiculescu's circular operator fail to have the UNZA property, and hence are examples of non-spectral operators. 
We can now extend this result to a broader class.
The idea is to use the upper-triangular model for DT-operators from Theorem~\ref{thm:dt-ut} and the trace estimates in Corollary~\ref{cor:Znorms},
along with the explicit form of Haagerup-Schultz projections
for annuli from Proposition~\ref{prop:hsproj-disc}, to construct arbitrarily close vectors belonging to complementary Haagerup-Schultz projections. 

In order to estimate the minimal angle of an operator, it is useful to cut it down by a projection. Restricting an operator to a Haagerup-Schultz projection can only increase the minimum angle between its subspaces.
\begin{lemma}\label{lem:sub-angle}
For an operator $T \in \Mcal$, and a Borel set $B\subset \Cpx$, let $p=P(T,B)$, and consider $Tp=pTp\in p\Mcal p$. Then, for a Borel set $C \subseteq \Cpx$,  
\[ 
\alpha_{p\Mcal p}(P(Tp,C),P(Tp,C^c)) \geq \alpha(P(T,C),P(T,C^c)), 
\]
so
\[ 
\alpha_{p\Mcal p}(Tp)\geq \alpha(T).
\] 
In particular, if $Tp$ fails to have the UNZA property, then so does $T$.
\begin{proof}
This follows from Theorem~\ref{prop:hsproj-sim} and Proposition~\ref{prop:hsproj-lattice}.
We have
\[ 
P_{p \Mcal p} (Tp, C) = P(T,C) \wedge P(T,B) = P(T,B\cap C) \leq P(T,C)
\] 
and 
\[ 
P_{p \Mcal p} (Tp, C^c) = P(T,C^c) \wedge P(T,B) = P(T,B\cap C^c) \leq P(T, C^c )
\] 
so the angles obey the inequality above.
\end{proof}
\end{lemma}

We are now ready to estimate the angles between some Haagerup-Schultz subspaces of DT-operators with suitable Brown measures.
We will frequently use the notation $A(r,s)$ for the closed annulus with radii $r<s$, centered at the origin.
When $r=0$, this equals the closed disk of radius $s$.

\begin{lemma}\label{lem:dtangle}
Let $0\leq r<r'<s'<s$ and $c>0$.
Let $Z$ be a $\DT(\mu,c)$-operator whose measure $\mu$ is radially symmetric and concentrated in the union of the annuli $A(r,r')$ and $A(s',s)$.
Let $t=\mu(A(r,r'))$, and assume $0<t<1$. 
Then 
\begin{equation}
\cos(\alpha(Z)) 
\geq \left(1+\frac{s^2-r^2}{c^2\max(t,1-t)} \right)^{-1/2}
\geq \left(1+\frac{2 (s^2-r^2)}{c^2} \right)^{-1/2}_. 
\label{eq:aZlbd}
\end{equation}
\end{lemma}
\begin{proof}
Let $\mu_1$ and $\mu_2$ denote the renormalized restrictions of $\mu$ to $A(r,r')$ and $A(s',s)$ respectively. 
From the upper-triangular model found in Theorem \ref{thm:dt-ut}, there exists
an example of a $\DT(\mu,c)$ operator
\[ Z = \left( \begin{matrix} Z_1 &  c p X (1-p) \\ 0 & Z_2 \end{matrix} \right), \]
where
$p \in \Mcal$ is a projection with $\tau(p)=t$,  $Z_1 \in p\Mcal p$ is a $\DT(\mu_1,c\sqrt t)$-operator, $Z_2 \in (1-p)\Mcal(1-p)$ is a $\DT(\mu_2,c\sqrt{1-t})$-operator and $X$ is a semicircular operator with $\tau(X^2)=1$
and so that $X$ and $\{Z_1,Z_2,p\}$ are $*$-free.

Regard $Z$ as an operator acting on the Hilbert space $\HEu=L^2(\Mcal) =L^2(\Mcal,\tau)$.
For $x \in \Mcal$, let $\hat{x}$ denote the corresponding vector in $L^2(\Mcal)$. 

Since DT operators have spectra equal to the supports of their Brown measures, $Z_2$ is invertible in $(1-p)\Mcal (1-p)$.
Let $\eps>0$. From the spectral radius formula, we have, for $k$ large enough,
\[ \| Z_1 ^k \| \leq (1+\eps) (r')^k , \qquad \|Z_2^{-k}\| \leq (1+\eps) (s')^{-k}_. \]
Since $r'<s'$, it follows that the series
\[
 \sum_{k=0}^n Z_1^k \left(c p X (1-p) \right) Z_2^{-k-1} 
\]
converges in operator norm to $Y \in \Mcal$. Since $pZ_1=Z_1$ and $Z_2(1-p)=Z_2$, a direct computation shows 
\[ S = \begin{pmatrix} p & Y \\ 0 & 1-p \end{pmatrix} \]
is invertible in $\Mcal$, with inverse
\[ S^{-1}=\begin{pmatrix} p & -Y \\ 0 & 1-p \end{pmatrix}_. \]
Moreover,
\begin{equation}\label{eq:SZSinv}
S \begin{pmatrix} Z_1 & 0 \\  0 & Z_2 \end{pmatrix} S^{-1} = \begin{pmatrix} Z_1 & cpX(1-p) \\ 0 & Z_2 \end{pmatrix}_.
\end{equation}

Since $\sigma(Z_1)\subseteq A(r,r')$ and $\sigma(Z_2)\subseteq A(s',s)$ (where the spectra are computed in the corners $p\Mcal p$ and $(1-p)\Mcal (1-p)$ respectively), we have 
\begin{gather*} 
P\left( \begin{pmatrix} Z_1 & 0 \\  0 & Z_2 \end{pmatrix}, A(r,r') \right) = p\HEu, \\
P\left( \begin{pmatrix} Z_1 & 0 \\  0 & Z_2 \end{pmatrix}, A(s',s) \right) = (1-p)\HEu.
\end{gather*}
Using equation \eqref{eq:SZSinv} and Proposition \ref{prop:hsproj-sim}, we get
\begin{equation}\label{eq:Arr'}
P(Z, A(r,r')) \HEu = S P\left( \begin{pmatrix} Z_1 & 0 \\  0 & Z_2 \end{pmatrix}, A(r,r') \right) \HEu = S (p\HEu) = p\HEu, 
\end{equation}
and
\begin{multline}
P(Z,A(s',s))\HEu = S P\left( \begin{pmatrix} Z_1 & 0 \\  0 & Z_2 \end{pmatrix}, A(s',s) \right) \HEu  = S(1-p)\HEu \\ = \left \{ \begin{pmatrix} Y \eta  \\ \eta \end{pmatrix}: \eta \in (1-p)\Mcal \right \}_.  \label{eq:As's}
\end{multline}
Let $\eta = \widehat{1-p}$. Then 
\[ Y\eta =  \sum_{k=0}^\infty Z_1^k cpX(1-p) Z_2^{-k-1} \widehat{1-p}_. \]
From the radial symmetry of $\mu_1$, it follows that $Z_1$ and $\lambda Z_1$ have the same $*$-moments for every complex number $\lambda$ of modulus $1$.
Thus, $\tau((Z_1^j)^*Z_1^k)=0$ whenever $k$ and $j$ are nonnegative integers with $j\ne k$.
Now using $*$-freeness of $X$ and $\{Z_1,Z_2,p\}$, we calculate, for $0\le m\le n$,
\begin{align}
&\left\| \sum_{k=0}^n Z_1^kcpX(1-p)Z_2^{-k-1}\widehat{(1-p)}\right\|_2^2 \notag \\ 
&=c^2 \sum_{0\le k_1,k_2\le n}\tau(\big(1-p)(Z_2^{-k_1-1})^*(1-p)Xp(Z_1^{k_1})^*Z_1^{k_2}pX(1-p)Z_2^{-k_2-1}\big) \notag \\
&=c^2 \sum_{0\le k_1,k_2\le n}\tau(1-p)\tau_{(1-p)}((Z_2^{-k_1-1})^*Z_2^{-k_2-1})\tau(X^2)\tau(p(Z_1^{k_1})^*Z_1^{k_2}p) \notag \\
&=c^2 \tau(p)\tau(1-p)\sum_{k=0}^n\tau_{(1-p)}((Z_2^{-k-1})^*Z_2^{-k-1})\tau_p((Z_1^k)^*Z_1^k). \label{eq:sumnorm}
\end{align}
From Corollary~\ref{cor:Znorms}
we have, for all $k\ge0$,
\begin{gather*}
\tau_{(p)} \left(  Z_1^{*k} Z_1^k \right)  \geq  r^{2k}, \\
\tau_{(1-p)} \left(  (Z_2^{*})^{-k-1} Z_2^{-k-1} \right) \geq s^{-2k-2}.
\end{gather*}
So using~\eqref{eq:sumnorm} we get
\begin{equation}
\| Y\eta \|_2^2 \geq c^2 t (1-t) \sum_{k=0}^\infty \frac{1}{s^2} \left(\frac{r}{s}\right)^{2k} = \frac{c^2 t(1-t)}{s^2 -r^2}. \label{eq:xiunder}
\end{equation}
From \eqref{eq:Arr'} and \eqref{eq:As's}, we have 
\[ \left( \begin{matrix}Y \eta \\ \eta \end{matrix} \right) \in P(Z,A(s',s)), \qquad \left( \begin{matrix} Y \eta \\ 0 \end{matrix} \right) \in P(Z,A(r,r')). \]
Computing the cosine of the angle between $ \left( \begin{matrix}Y \eta \\ \eta \end{matrix} \right)$ and $ \left( \begin{matrix} Y \eta \\ 0 \end{matrix} \right)$ gives 
\begin{multline*}
\cos \left( \alpha(P(Z,A(r,r')),P(Z,A(s',s)) \right) \geq \frac{ \| Y \eta \| }{((1-t)+\| Y \eta \|^2)^{1/2}} \\
=\left(1+\frac{1-t}{\|Y \eta\|^2}\right)^{-1/2}.
\end{multline*}
Now using the lower bound~\eqref{eq:xiunder}, we get
\begin{equation}\label{eq:cosalphat}
\cos(\alpha(Z)) \geq \left(1+\frac{s^2-r^2}{c^2 t}\right)^{-1/2}_.
\end{equation}

In a similar fashion we find also an example of a $\DT(\mu,c)$ operator 
\[ Z = \left( \begin{matrix} Z_2 & c (1-p) X p \\ 0 & Z_1 \end{matrix} \right), \]
where
$p \in \Mcal$ is a projection with $\tau(1-p)=1-t$,  $Z_1 \in p\Mcal p$ is a $\DT(\mu_1,c \sqrt t)$-operator, $Z_2 \in (1-p)\Mcal(1-p)$ is a $\DT(\mu_2,c \sqrt{1-t})$-operator and $X$ is a semicircular operator with $\tau(X^2)=1$
and so that $X$ and $\{Z_1,Z_2,p\}$ are $*$-free.
With $Y= -\sum_{k=0}^\infty Z_2^{-k-1} (1-p)Xp Z_1^k$, and
\[
S = \begin{pmatrix} (1-p) & Y \\ 0 &  p \end{pmatrix}_,
\]
we have 
\[ 
S \begin{pmatrix} Z_2 & 0 \\ 0 & Z_1 \end{pmatrix} S^{-1} = Z
\]
so by similar reasoning, we get
\[ \left( \begin{matrix}Y \widehat{p} \\ \widehat{p} \end{matrix} \right) \in P(Z,A(r,r')), \qquad \left( \begin{matrix} Y \widehat{p} \\ 0 \end{matrix} \right) \in P(Z,A(s',s)). \]
Then this gives us the estimate
\begin{equation}\label{eq:cosalpha1-t}
\cos(\alpha(Z)) \geq \left(1+\frac{s^2-r^2}{c^2 (1-t)}\right)^{-1/2}_.
\end{equation}
Combining equations \eqref{eq:cosalphat} and \eqref{eq:cosalpha1-t} gives us the required bound.
\end{proof}

This lemma allows us to show that many $\DT$-operators are not spectral.
For ease of notation, when $r\le0$ and $s>0$, $A(r,s)$ will denote the closed disc $A(0,s)$, and $A(s,s)$ will denote the circle of radius $s$.

\begin{thm}\label{thm:unza_measures}
Let $c>0$, and let $\mu$ be a radially symmetric, compactly supported Borel probability measure on $\Cpx$, such that there exists $x_0\geq 0$ with
\[  
\lim_{\delta \to 0^+}  \frac{ \mu(A(x_0-\delta,x_0+\delta)\setminus A(x_0,x_0))}{\delta} = \infty.
\]
Let $Z$ be a $\DT(\mu,c)$ operator. Then $Z$ fails to satisfy the UNZA property, and hence is not spectral.
\end{thm}
\begin{proof}
Let $N \in \mathbb{N}$. Choose $\eps>0$ such that for all $\delta\in(0,\eps]$,
\begin{equation}\label{eq:muA} \frac{\mu(A(x_0-\delta,x_0+\delta)\setminus A(x_0,x_0))}{\delta } > N. \end{equation}
Let $r =\max\{0,x_0-\eps\}$ and $s=x_0+\eps$. 
Since $\mu(A(x_0-\delta,x_0+\delta))>0$ for all $\delta\leq \eps$, we have 
\[ 
\sup\{ x \in (r,s): \mu(A(r,x))=0 \} \leq x_0
\]
and
\[
\inf\{ x \in (r,s): \mu(A(x,s))=0 \} \ge x_0. \]
Further, these two quantities cannot both be equal to $x_0$, as this would violate equation \eqref{eq:muA}. Using this fact, and the fact that the set of circles centered at $0$ with positive $\mu$-measure is countable, we may choose $x_\eps \in (r,s)$ such that
$\mu(A(r,x_\eps))\neq 0$, $\mu(A(x_\eps,s))\neq 0$, and $\mu(A(x_\eps,x_\eps))=0$. 
From~\eqref{eq:muA} with $\delta=\eps$, we have $\mu(A(r,s)\setminus A(x_0,x_0))>\eps N$.
Since $\mu(A(x_\eps,x_\eps))=0$, we may now choose $r',s'$, with $r<r'<x_\eps<s'<s$, such that 
\[ \mu(A(r,r'))>0 ,\quad \mu(A(s',s))>0, \]
and
\begin{equation}\label{eq:muB}
\mu(A(r,r')\cup A(s',s)) > \eps N.
\end{equation}
Let $B$ be the set $A(r,r')\cup A(s',s)$, $q=P(Z,B)$ and let $\tilde{Z}$ be the restriction $Zq$. Let $\tilde{\mu}$ be the renormalized restriction of $\mu$ to $B$.
From Theorem \ref{thm:dt-properties}, $Zq$ is a $\DT(\tilde{\mu},c\sqrt{\mu(B)})$ operator.
Applying Lemmas~\ref{lem:dtangle} and~\ref{lem:sub-angle} now gives us the estimate
\[ \cos(\alpha(Z)) \ge \cos(\alpha(\tilde{Z})) \ge \left(1 + \frac{2(s^2-r^2)}{c^2 \mu(B)} \right)^{-1/2}_.\]
Since $x_0$ is in the support of $\mu$ we have $x_0 \leq \norm{Z}$.
Consequently, when $x_0-\eps>0$, we get $ s^2 - r^2 = 4 x_0 \eps \leq 4 \eps \norm{Z}$.
On the other hand, if $x_0-\eps\le 0$, then
\[ s^2 - r^2 = (x_0+\eps)^2 -0 \le 4 \eps^2 \le 4 \eps \norm{Z}, 
\]
since we may safely assume that $\eps<\norm{Z}$.
Hence, using~\eqref{eq:muB}, we get
\[ \frac{2(s^2-r^2)}{c^2 \mu(B)} \le \frac{8 \norm{Z}}{c^2} \frac{1}{N}_. \]
Letting $N$ be arbitrarily large shows that $\cos(\alpha(Z))$ is arbitrarily close to $1$, which implies $\alpha(Z)=0$, and $Z$ fails to have the UNZA property.

\end{proof}

There are many examples of measures which satisfy the concentration hypothesis of Theorem \ref{thm:unza_measures}.

\begin{example}
Let $a \geq 0$. Let $\mu$ be a  radially symmetric Borel probability measure supported on the union of the circles $C_n$ of radius $a+1/n$ ($n \in \Nats$), with $\mu(C_n)=w_n$. If 
\[ \lim_{n \to \infty} n \sum_{k=n}^\infty w_k = \infty,\]
then $\mu$ satisfies the hypothesis of Theorem \ref{thm:unza_measures}, so if $Z$ is a $\DT(\mu,1)$ operator, $Z$ is non-spectral. This happens, for instance, when $w_n$ is asymptotically proportional to $n^{-b}$, for some $1<b<2$.
\end{example}

\begin{example}
Let $f$ be a non-zero, compactly supported, Lebesgue integrable function on $[0,\infty)$. Assume that there is a point $x_0 \in [0,\infty)$ such that  
\[ \lim_{\eps \to 0} \frac{1}{\eps}  \int_{x_0-\eps}^{x_0+\eps} |f(x)|\, dx = \infty.
\]
Let $\mu$ be a radially symmetric measure on $\Cpx$ defined by 
\[ \mu(A(r,s)) = \frac{1}{\norm{f}_1} \int_r^s |f(x)|\, dx.\]
Then $\mu$ satisfies the hypothesis of Theorem~\ref{thm:unza_measures}. For instance, this happens with $x_0=0$
when $f(x)=x^{-a}$ near $0$, for some $0<a<1$.
\end{example}

\begin{rem}\label{rmk:cFP}
It is interesting that the hypothesis of Theorem~\ref{thm:unza_measures} does not hold for the circular free Poisson operators,
whose Brown measures are uniform (Lebesgue) measures on their supports, which are disks or annuli centered at the origin.
The proof, found in~\cite{DKU19}, that the circular free Poisson operators fail to satisfy the UNZA property depended on quite precise
knowledge about certain $*$-moments of the circular free Poisson operators and their inverses.
Comparing the result for circular free Poisson operators with Theorem~\ref{thm:unza_measures}, we have
the coincidence of failure of the UNZA property for DT-operators whose Brown measures have quite diverse behaviors, by somewhat different methods of proof.
This suggests that failure of the UNZA property might occur for many more DT-operators than those treated here or in~\cite{DKU19}.
\end{rem}

\end{document}